\documentclass[11pt]{article}
\usepackage[T1]{fontenc}
\usepackage[english]{babel}
\usepackage{microtype}
\usepackage{amsmath,amssymb,amsfonts,amsthm,mathtools}
\usepackage[shortlabels]{enumitem}
\usepackage{geometry}
\usepackage{xcolor}
\usepackage[colorlinks=true,linkcolor=blue,citecolor=blue,urlcolor=blue]{hyperref}
\newtheorem{theorem}{Theorem}[section]
\newtheorem{lemma}[theorem]{Lemma}
\newtheorem{proposition}[theorem]{Proposition}
\newtheorem{corollary}[theorem]{Corollary}
\newtheorem{observation}[theorem]{Observation}
\theoremstyle{remark}
\newtheorem{remark}[theorem]{Remark}

\newcommand{\tr}{\operatorname{tr}}
\newcommand{\dist}{\operatorname{dist}}

\title{Sharp edge-spectral supersaturation for odd cycles}
\author{{ Jiaqi Liu$^a$}\thanks{ School of Mathematics, University of Shanghai for Science and Technology, Shanghai 200093, China. E-mail address: m19821217592@163.com},\ \ { Zhenzhen Lou$^a$}\thanks{Corresponding author. School of Mathematics, University of Shanghai for Science and Technology, Shanghai 200093, China and Extremal Combinatorics and Probability Group (ECOPRO), Institute for Basic Science (IBS), Daejeon, South Korea.  E-mail address: louzz@usst.edu.cn. 
},\ \
{ Shuang Sun$^b$}\thanks{School of Mathematical Sciences, Shanghai Jiao Tong University, Shanghai 200240, China. E-mail address: chocolatesun@sjtu.edu.cn.}
}
\date{}

\begin{document}
\maketitle
\vspace{-0.5cm}

\begin{abstract}
Let \(G\) be a graph with \(m\) edges and adjacency spectral radius
\(\rho(G)\), and let \(N(C_{2k+1},G)\) denote the number of copies of
\(C_{2k+1}\) in \(G\). For each fixed integer \(k\ge 2\), define
\(
g_k(m):=\frac{k-1+\sqrt{4m-k^2+1}}{2}.
\)
Li, Zhai and Shu [European J. Combin., 2024] determined the spectral extremal threshold for odd
cycles by proving that, for all sufficiently large \(m\), every
\(C_{2k+1}\)-free graph \(G\) with \(m\) edges satisfies
\(\rho(G)\le g_k(m)\).

We establish the asymptotically sharp supersaturation counterpart of
their result. More precisely, for every fixed integer \(k\ge 2\), we
prove that
\[
\inf_{\substack{e(G)=m\\ \rho(G)>g_k(m)}}
\frac{N(C_{2k+1},G)}{m^k}
=
\frac{\lceil k^2/2\rceil (k-1)!}{(k+1)^k}+o(1)
\qquad\text{as }m\to\infty.
\]
Thus every \(m\)-edge graph whose spectral radius exceeds the
\(C_{2k+1}\)-free threshold contains at least
\[
\Big(
\frac{\lceil k^2/2\rceil (k-1)!}{(k+1)^k}-o(1)
\Big)m^k
\]
copies of \(C_{2k+1}\), and the leading constant is asymptotically
best possible.

In particular, taking \(k=2\), we obtain if
\(
\rho(G)>\frac{1+\sqrt{4m-3}}{2}\)
then 
\(
N(C_5,G)\ge \left(\frac{2}{9}-o(1)\right)m^2,
\)
with the constant \(2/9\) being asymptotically optimal. This answers
a question of Chen, Li and Tang concerning the existence and the
largest possible value of a constant \(C>0\) for which the same
spectral condition guarantees at least \(Cm^2\) copies of \(C_5\).
More generally, our result resolves a recent problem of Li, Lin, Liu
and Zhang on spectral supersaturation for odd cycles. The proof
combines spectral stability and resolvent analysis with estimates for
odd spectral moments and a careful treatment of non-injective closed
walks.


\end{abstract}

\section{Introduction}
Spectral extremal graph theory asks how large an eigenvalue of a graph can be
under prescribed combinatorial restrictions.  In the fixed-size setting,
where the number of edges rather than the number of vertices is prescribed,
a basic problem of Brualdi--Hoffman--Tur\'an type is to determine
$$
   \max\{\rho(G):e(G)=m,\ H\nsubseteq G\}
$$
for a fixed graph $H$, where $\rho(G)$ denotes the adjacency spectral
radius of $G$.  Once the corresponding extremal threshold is known, a
natural next question is one of \emph{spectral supersaturation}: how many
copies of $H$ must occur when the spectral radius is strictly above this
threshold?
This paper answers this question asymptotically, with
the best possible leading constant, for every odd cycle.

Supersaturation originates from classical extremal graph theory.  Rather
than merely guaranteeing the existence of a forbidden graph once an
extremal threshold is exceeded, one asks for the minimum number of copies
that must occur.  The general theory was developed by Erd\H{o}s and
Simonovits \cite{Erdos-Simonovits-1983}, and sharp counting results are known
for several important classes of graphs; see, for example, Mubayi
\cite{Mubayi-2010} and Pikhurko and Yilma
\cite{Pikhurko-Yilma-2017} for color-critical graphs.

A spectral version replaces edge excess by an eigenvalue condition.
Foundational results of Nikiforov
\cite{Nikiforov-2002,Nikiforov-2009,Nikiforov-2009-saturation}
and Bollob\'as and Nikiforov \cite{Bollobas-Nikiforov-2007} showed that
spectral thresholds not only force the occurrence of prescribed subgraphs,
but can also force quantitatively many copies.  More recently, this
viewpoint has led to sharp spectral supersaturation results in both the
fixed-order and fixed-size settings; see, for instance,
\cite{Li-Feng-Peng-2025a,Li-Feng-Peng-2025b,Li-Lu-Peng-2024,
Fang-Li-Lin-Ma-2025}.

The fixed-size, or edge-spectral, setting has a rather different geometry
from the fixed-order setting.  Here isolated vertices are immaterial and
the relevant extremal graphs are often highly unbalanced.  The first basic
threshold is Nosal's inequality: an $m$-edge triangle-free graph satisfies
$\rho(G)\le\sqrt m$.  At this threshold, Ning and Zhai
\cite{Ning-Zhai-2023} obtained sharp counting results for triangles and
initiated the corresponding study of four-cycles
\cite{Ning-Zhai-2025}. 
A substantial extension from $C_4$ to arbitrary even cycles was recently
obtained by Li, Lin, Liu and Zhang
\cite{Li-Lin-Liu-Zhang-2026}.  They developed spectral Sidorenko
inequalities and applied them to edge-spectral supersaturation for complete
bipartite graphs and even cycles.  Writing $S_{k-1,m}$ for the corresponding
$m$-edge split graph, they proved, for every fixed $k\ge2$, that if
$
   \rho(G)>\rho(S_{k-1,m}),
$
then
$$
   N(C_{2k},G)
   \ge
   \left(
      \frac{(k-1)!}{2k^k}-o(1)
   \right)m^k.
$$
The constant $(k-1)!/(2k^k)$ is asymptotically best possible.  They also
obtained sharp asymptotic edge-spectral supersaturation results for complete
bipartite graphs $K_{k,k}$.  Thus the sharp leading coefficient is now
understood for a broad family of bipartite graphs, including all even cycles.

Sharp counting results are also available at the opposite end of the
chromatic spectrum.  Let $F$ be a color-critical graph of order $f$ with
$
   \chi(F)=r+1\ge4.
$
Li, Liu and Zhang \cite{Li-Liu-Zhang-2025-turan} established the corresponding
edge-spectral Tur\'an theorem: every sufficiently large $m$-edge
$F$-free graph satisfies
$$
   \rho(G)\le
   \sqrt{2\big(1-\frac1r\big)m},
$$
with equality precisely for regular complete $r$-partite graphs.
Fang, Lin and Zhai \cite{Fang-Lin-Zhai-2026} subsequently determined the
asymptotically sharp number of copies forced at this threshold.  To state
their result, let $c(n,F)$ be the minimum number of copies of $F$ created
by adding one edge to the Tur\'an graph $T_{n,r}$, and write
$
   c(n,F)
   =
   \bigl(\alpha_F+o(1)\bigr)n^{f-2}.
$
They proved that if
$
   \rho(G)\ge
   \sqrt{2\big(1-\frac1r\big)m},
$
then, unless $G$ is a regular complete $r$-partite graph,
$$
   N(F,G)
   \ge
   \Big[
      \Big(\frac{2r}{r-1}\Big)^{(f-2)/2}
      \alpha_F-o(1)
   \Big]
   m^{(f-2)/2},
$$
and both the exponent and the leading constant are best possible.
In particular, for $F=K_{r+1}$,
$$
   N(K_{r+1},G)
   \ge
   \Big[
      \Big(\frac{2}{r(r-1)}\Big)^{(r-1)/2}
      -o(1)
   \Big]
   m^{(r-1)/2}.
$$

More recently, Chen and Li \cite{Chen-Li-2026} obtained a sharp
surplus-dependent refinement, which may be viewed as an edge-spectral
counterpart of Mubayi's supersaturation theorem.  For every color-critical
$F$ with $\chi(F)=r+1\ge4$, there exists $\delta_F>0$ such that, whenever
$
   0<q\le\delta_F\sqrt m$
and $
   \rho(G)^2
   \ge
   2\left(1-\frac1r\right)m+q,
$
one has
$$
   N(F,G)
   \ge
   (B_F-o(1))q\,m^{(f-2)/2},
   \qquad
   B_F:=
   \frac{\alpha_F}{4}
   \Big(\frac{2r}{r-1}\Big)^{f/2},
$$
where $B_F$ is best possible.

These developments leave a particularly interesting gap at chromatic number
three.  Although odd cycles are among the most fundamental color-critical
graphs, the general sharp counting theory for color-critical graphs with
chromatic number at least four does not extend directly to them.  Indeed,
their edge-spectral extremal threshold has a different form and is governed
by highly unbalanced split graphs rather than regular multipartite graphs.
This makes the three-chromatic case substantially more delicate.

Fix an integer $k\ge2$ and define
$
   g_k(m):=
   \frac{k-1+\sqrt{4m-k^2+1}}{2}.
$
Zhai, Lin and Shu \cite{Zhai-Lin-Shu-2021} determined the maximum spectral
radius of $C_5$-free graphs with $m$ edges, and Li, Zhai and Shu
\cite{Li-Zhai-Shu-2024-cycles} extended their result to
$C_{2k+1}$ for every fixed $k\ge3$.  In particular, for every fixed
$k\ge2$ and all sufficiently large $m$, if
$
   C_{2k+1}\nsubseteq G$,
then $
   \rho(G)\le g_k(m).
$
Whenever $m-\binom{k}{2}$ is divisible by $k$, equality is attained by
$
   K_k\vee I_{(m-\binom{k}{2})/k}.
$
Hence $g_k(m)$ is the natural edge-spectral threshold for the appearance
of $C_{2k+1}$.

Very recently, Fang, Lin, and Zhai \cite{Fang-Lin-Zhai-2026-tripartite} initiated a systematic study of edge‑spectral supersaturation in this delicate 3-chromatic regime. As a corollary of their main results, they proved that for any fixed integer $k\ge 2$, if
$
\rho(G)>g_k(m),
$
then
$
N(C_{2k+1},G)=\Omega_k(m^k),
$
and they further demonstrated that the exponent $k$ is sharp.
Thus the correct polynomial
order is known, but the sharp leading coefficient remains undetermined.

Motivated by these results and by the sharp theory for even cycles and
higher-chromatic color-critical graphs, Li, Lin, Liu and Zhang
\cite{Li-Lin-Liu-Zhang-2026} posed the problem of determining the sharp
asymptotic value of
$$
   \inf_{\substack{e(G)=m\\ \rho(G)>g_k(m)}}
   \frac{N(C_{2k+1},G)}{m^k}.
$$
The main result of the present paper resolves this problem for every
$k\ge2$.

We determine this coefficient for every odd cycle as follows.

\begin{theorem}\label{thm:main}
For every fixed integer $k\ge2$,
$$
   \inf_{\substack{e(G)=m\\ \rho(G)>g_k(m)}}
   \frac{N(C_{2k+1},G)}{m^k}
   =
   \frac{\lceil k^2/2\rceil (k-1)!}{(k+1)^k}+o(1)
   \qquad\text{as }m\to\infty.
$$
Equivalently, for every $\varepsilon>0$, there exists
$m_0=m_0(k,\varepsilon)$ such that every graph $G$ with
$e(G)=m\ge m_0$ and
$
\rho(G)>g_k(m)
$
satisfies
$
N(C_{2k+1},G)
\ge
\left(
\frac{\lceil k^2/2\rceil (k-1)!}{(k+1)^k}
-\varepsilon
\right)m^k.
$
Moreover, the constant
$\frac{\lceil k^2/2\rceil (k-1)!}{(k+1)^k}$
is asymptotically best possible.
\end{theorem}

Taking \(k=2\) in Theorem~\ref{thm:main}, we obtain the following
asymptotically sharp answer to the question of Chen, Li and Tang in \cite{Chen-Li-Tang}.

\begin{corollary}\label{cor:C5}
Every sufficiently large \(m\)-edge graph $G$ satisfying
\(
\rho(G)>\frac{1+\sqrt{4m-3}}{2}
\)
contains at least
\(
\left(\frac{2}{9}-o(1)\right)m^2
\)
copies of \(C_5\), and the constant \(2/9\) is asymptotically best
possible.
\end{corollary}

Theorem~\ref{thm:main} therefore completes, for odd cycles, the
sharp-constant picture suggested by the recent edge-spectral
supersaturation theory for even cycles.  A notable distinction between the
two settings is that the extremal threshold for an odd cycle contains a
constant-order shift from the bipartite scale $\sqrt m$, reflecting the
presence of a small non-bipartite core.  Moreover, spectral moment methods
naturally count closed walks, whereas $N(C_{2k+1},G)$ counts injective
cycles.  Controlling the contribution of non-injective odd closed walks at
the precise $m^k$-scale is one of the main difficulties in obtaining the
sharp coefficient.

We briefly describe the proof.  A spectral supersaturation-stability theorem
of Fang, Lin and Zhai \cite{Fang-Lin-Zhai-2026} first reduces a putative
counterexample to a graph whose edge set differs from that of a complete
bipartite graph in $o(m)$ edges.  We then distinguish two regimes according
to the sizes of the two bipartition classes.
When one bipartition class remains bounded, the spectral radius is sensitive
to the finite non-bipartite core.  A Schur-complement expansion at the
critical value $g_k(m)$ identifies $K_k$ as the unique critical core;
every competing configuration forces sufficiently many copies of
$C_{2k+1}$.
When both bipartition classes tend to infinity, the
$(2k+1)$-spectral moment becomes more effective.  The quantity
$
   \tr\bigl(A(G)^{2k+1}\bigr)
$
counts all closed walks of length $2k+1$, including those with repeated
vertices.  Using the near-bipartite structure together with an edge-cover
estimate, we prove that non-injective closed walks contribute only
$o(m^k)$.  The spectral moment can therefore be converted, at the required
precision, into a lower bound for the number of copies of $C_{2k+1}$.
Combining the two regimes yields the lower bound in
Theorem~\ref{thm:main}, while an explicit construction proves that $\frac{\lceil k^2/2\rceil (k-1)!}{(k+1)^k}$
cannot be improved.

The remainder of the paper is organized as follows. Section~2 contains
the preliminaries and the sharp construction. Section~3 treats the two
cases after the stability reduction. Section~4 completes the proof of
the main theorem.

\section{Preliminaries}
Throughout the paper, all graphs are finite and simple. For a graph $G$,
we write $V(G)$, $E(G)$, and $e(G)$ for its vertex set, edge set, and
number of edges, respectively. Let $A(G)$ denote the adjacency matrix of
$G$, and let $\rho(G)$ be its spectral radius. For graphs $F$ and $G$,
we denote by $N(F,G)$ the number of unoriented copies of $F$ in $G$.
For $X\subseteq V(G)$ and $v\in V(G)$, write
$
N_X(v):=N_G(v)\cap X
$
and
$
d_X(v):=|N_X(v)|.
$

When two graphs $G$ and $H$ are regarded as graphs on a common vertex
set, with isolated vertices added if necessary, we define
$
   \dist(G,H):=|E(G)\triangle E(H)|.
$
For disjoint sets $A,B\subseteq V(G)$, let $K_{A,B}$ denote the complete
bipartite graph with parts $A$ and $B$, viewed as a graph on $V(G)$
with all vertices outside $A\cup B$ isolated. As usual, $K_{a,b}$
denotes the complete bipartite graph with part sizes $a$ and $b$.

A closed walk
$
   v_0v_1\cdots v_{j-1}v_0
$
is called \emph{non-injective} if the vertices
$v_0,v_1,\ldots,v_{j-1}$ are not pairwise distinct. An \emph{edge cover} of a graph $G$ is a set
	$\mathcal{F}\subseteq E(G)$ such that every vertex of $G$ is incident
	with at least one edge in $\mathcal{F}$. Subscripts on
$O(\cdot)$, $\Omega(\cdot)$, and positive constants indicate the
parameters on which the implied constants may depend. Unless otherwise
specified, all $o(1)$-terms are taken with respect to the asymptotic
parameter appearing in the relevant statement.

We shall also use the following stability theorem. For a graph $F$, let
$\beta'(F)$ denote the minimum size of an independent vertex cover of
$F$.

\begin{lemma}[Fang--Lin--Zhai \cite{Fang-Lin-Zhai-2026}, Theorem~1.3]
\label{thm:flz-stability}
Let $F$ be a fixed graph of order $f$ with $\chi(F)=r+1\ge2$, and let
$G$ be a sufficiently large $m$-edge graph satisfying
$N(F,G)=o(m^{f/2})$. For every $\varepsilon>0$, there exists
$\delta>0$ such that the following statements hold.
\begin{enumerate}
\item[(i)] If $r\ge3$ and
$\rho(G)\ge\sqrt{2m(1-1/r-\delta)}$, then there is an $r$-partite
Tur\'an graph $T_{n,r}$ on $n$ vertices, with
$V(T_{n,r})\subseteq V(G)$ and
$\dist(G,T_{n,r})\le\varepsilon m$.
\item[(ii)] If $\rho(G)\ge\sqrt{(1-\delta)m}$ and either $r=2$, or
$r=1$ and $\beta'(F)\ge2$, then there are disjoint sets
$U,V\subseteq V(G)$ such
that $\dist(G,K_{U,V})\le\varepsilon m$.
\end{enumerate}
\end{lemma}

We first derive the sequence form of
Lemma~\ref{thm:flz-stability} used below.

\begin{lemma}\label{lem:stability}
Let $G_n$ be a sequence of $m_n$-edge graphs with $m_n\to\infty$,
$$
N(C_{2k+1},G_n)=o(m_n^{k+1/2}),
 \qquad
 \rho(G_n)\ge\sqrt{(1-o(1))m_n}.
$$
Then, after passing to a subsequence, there are disjoint sets
$A_n,B_n\subseteq V(G_n)$ such that
$$
 \dist(G_n,K_{A_n,B_n})=o(m_n).
$$
\end{lemma}
\begin{proof}
Apply Lemma~\ref{thm:flz-stability}\textup{(ii)} with
$F=C_{2k+1}$. Here $\chi(F)=3$, so $r=2$, and
$|F|/2=k+1/2$. For each $j\ge1$, take $\varepsilon=1/j$ and let
$\delta_j>0$ be supplied by that lemma. Since
$N(C_{2k+1},G_n)=o(m_n^{k+1/2})$ and
$\rho(G_n)\ge\sqrt{(1-o(1))m_n}$, we may choose
$n_j>n_{j-1}$ sufficiently large so that all the hypotheses of
Lemma~\ref{thm:flz-stability}\textup{(ii)} hold with
$\varepsilon=1/j$ and $\delta=\delta_j$. Hence there are disjoint
sets $A_{n_j},B_{n_j}\subseteq V(G_{n_j})$ satisfying
$$
\dist(G_{n_j},K_{A_{n_j},B_{n_j}})
\le\frac{m_{n_j}}{j}.
$$
Relabelling this subsequence gives the desired conclusion.
\end{proof}

\begin{observation}\label{obs:threshold-optimization}
The following two statements hold.
\begin{enumerate}
\item[(i)] For every $m$ for which $g_k(m)$ is defined,
$
g_k(m)^2-(k-1)g_k(m)=m-\binom{k}{2}.
$
Moreover, for fixed $k$,
$$
g_k(m)=\sqrt m+\frac{k-1}{2}+O_k(m^{-1/2});
$$

\item[(ii)]  For integers $k\ge2$ and $r\ge k+1$,
$$
\frac{\left(\left\lfloor (k-1)r/2\right\rfloor+1\right)
(k-1)!\binom{r-2}{k-1}}{r^k}
\ge
\frac{\lceil k^2/2\rceil (k-1)!}{(k+1)^k}.
$$
\end{enumerate}
\end{observation}

\begin{proof}
For (i), the first identity follows by squaring
$
2g_k(m)-(k-1)=\sqrt{4m-k^2+1}.
$
Also,
$
\sqrt{4m-k^2+1}=2\sqrt m+O_k(m^{-1/2}),
$
which gives the stated expansion.

For (ii), first let $r=k+1$. Then
$
\left\lfloor\frac{(k-1)(k+1)}{2}\right\rfloor+1
=
\left\lceil\frac{k^2}{2}\right\rceil,
$
and
$
(k-1)!\binom{k-1}{k-1}=(k-1)!,
$
so equality holds.
Now let $r\ge k+2$. Since
$
\left\lfloor (k-1)r/2\right\rfloor+1\ge (k-1)r/2,
$
we have
$$
\begin{aligned}
\frac{\left(\left\lfloor (k-1)r/2\right\rfloor+1\right)
(k-1)!\binom{r-2}{k-1}}{r^k}
&\ge
\frac{k-1}{2}
\prod_{j=2}^{k}\Big(1-\frac{j}{r}\Big)
\ge
\frac{(k-1)k!}{2(k+2)^{k-1}}.
\end{aligned}
$$
It remains to show
$$
\frac{(k-1)k!}{2(k+2)^{k-1}}
\ge
\frac{\lceil k^2/2\rceil (k-1)!}{(k+1)^k}.
$$
Equivalently,
$$
\frac{k(k-1)(k+2)}{2\lceil k^2/2\rceil}
\left(\frac{k+1}{k+2}\right)^k
\ge1.
$$
By Bernoulli's inequality,
$
((k+1)/(k+2))^k\ge2/(k+2),
$
and hence the left-hand side is at least
$
\frac{k(k-1)}{\lceil k^2/2\rceil}\ge1.
$
This proves (ii).
\end{proof}

\begin{lemma}[Weyl's inequality~\cite{Horn-Johnson-2013}]
\label{lem:weyl}
Let $M$ and $N$ be real symmetric matrices of order $n$, with eigenvalues
$\lambda_1(M)\ge\cdots\ge\lambda_n(M)$ and
$\lambda_1(N)\ge\cdots\ge\lambda_n(N)$, respectively. Then, for every
$1\le i\le n$,
$$
 |\lambda_i(M)-\lambda_i(N)|
 \le \|M-N\|\le\|M-N\|_F,
$$
where $\|\cdot\|$ and $\|\cdot\|_F$ denote the operator norm and
the Frobenius norm, respectively.
\end{lemma}

\begin{lemma}[de Caen~\cite{deCaen}]\label{lem:decaen}
If $G$ is a graph with $n\ge2$ vertices and $m$ edges, then
$$
 \sum_{u\in V(G)}d_G(u)^2
 \le m\Big(\frac{2m}{n-1}+n-2\Big).
$$
\end{lemma}

We conclude this section by showing that the constant in
Theorem~\ref{thm:main} is asymptotically best possible.

\begin{proposition}\label{prop:construction}
For every fixed $k\ge2$ and every sufficiently large integer $m$, there
is an $m$-edge graph $G$ satisfying $\rho(G)>g_k(m)$ and
$$
N(C_{2k+1},G)=
\left(
\frac{\lceil k^2/2\rceil (k-1)!}{(k+1)^k}
+o(1)
\right)m^k.
$$
\end{proposition}

\begin{proof}
	For every sufficiently large $m$, write
	$m=(k+1)t+\lceil k^2/2\rceil+r$, where $0\le r<k+1$.
	Take any graph $H$ on $k+1$ vertices with
	$\lceil k^2/2\rceil$ edges, and let
	$
	G_m=(H\vee I_t)\cup rK_2,
	$
	where the $r$ copies of $K_2$ are disjoint from $H\vee I_t$. Then
	$e(G_m)=m$ and $\rho(G_m)=\rho(H\vee I_t)$. 
Let $X_t=H\vee I_t$. Define $\mathbf{x},\mathbf{y}\in
\mathbb R^{V(X_t)}$ by
$$
\mathbf{x}_v=
\begin{cases}
(k+1)^{-1/2},& v\in V(H),\\
0,& v\in V(I_t),
\end{cases}
\qquad
\mathbf{y}_v=
\begin{cases}
0,& v\in V(H),\\
t^{-1/2},& v\in V(I_t).
\end{cases}
$$
Then $\mathbf{x}$ and $\mathbf{y}$ are orthonormal. Let
	$$
	Q:=
	\begin{pmatrix}
		\mathbf{x}^{\top}A(X_t)\mathbf{x}&
		\mathbf{x}^{\top}A(X_t)\mathbf{y}\\
		\mathbf{y}^{\top}A(X_t)\mathbf{x}&
		\mathbf{y}^{\top}A(X_t)\mathbf{y}
	\end{pmatrix}
	=\begin{pmatrix}
		2\lceil k^2/2\rceil/(k+1)&\sqrt{(k+1)t}\\
		\sqrt{(k+1)t}&0
	\end{pmatrix}.
	$$
	For any $\alpha,\beta\in\mathbb R$ with
	$\mathbf z:=\alpha\mathbf{x}+\beta\mathbf{y}\ne\mathbf0$, orthonormality gives
	$$
	\frac{\mathbf z^{\top}A(X_t)\mathbf z}{\mathbf z^{\top}\mathbf z}
	=\frac{(\alpha,\beta)Q(\alpha,\beta)^{\top}}
	{\alpha^2+\beta^2}.
	$$
	The maximum of the right-hand side is $\lambda_{\max}(Q)$, whereas
	$\rho(X_t)$ is the maximum Rayleigh quotient over all nonzero vectors.
	Thus
	$
	\rho(G_m)=\rho(X_t)\ge
	\lambda_{\max}(Q)
	$
	and
	$\lambda_{\max}(Q)=
	(
	2\lceil k^2/2\rceil/(k+1)
	+
	\sqrt{
		\left(2\lceil k^2/2\rceil/(k+1)\right)^2
		+4(k+1)t
	})/2.
	$
	Let
	$$
	L_t:=
	\frac{
		2\lceil k^2/2\rceil/(k+1)
		+
		\sqrt{
			\left(2\lceil k^2/2\rceil/(k+1)\right)^2
			+4(k+1)t
		}
	}{2}.
	$$
	Then $\rho(G_m)\ge L_t$, and
	$$
	L_t=
	\sqrt{(k+1)t}
	+\frac{\lceil k^2/2\rceil}{k+1}
	+O_k(t^{-1/2}),
	$$
	whereas Observation~\ref{obs:threshold-optimization}(i) gives
	$$
	g_k(m)=\sqrt{(k+1)t}+(k-1)/2+O_k(t^{-1/2}).
	$$ Since
	$\lceil k^2/2\rceil/(k+1)>(k-1)/2$, we have
	$L_t>g_k(m)$ for all sufficiently large $t$. Hence
	$\rho(G_m)>g_k(m)$.
	
	The $r$ components isomorphic to $K_2$ contain no cycles.  Since $I_t$
	is independent, a copy of $C_{2k+1}$ in $X_t$ contains at most $k$
	vertices of $I_t$.  If it contains exactly $k$, then its other $k+1$
	vertices are all the vertices of $H$, and exactly one of its edges lies
	in $H$.  Orient each edge of $H$ once.  Choosing this edge, choosing and
	ordering the $k$ vertices of $I_t$, and ordering the remaining $k-1$
	vertices of $H$ gives $\lceil k^2/2\rceil k!(k-1)!\binom tk$ copies,
	each counted once because its unique edge in $H$ determines the
	chosen edge and its fixed orientation determines the ordering.  Copies
	using at most $k-1$ vertices of $I_t$ contribute $O_k(t^{k-1})$: their
	vertices in $I_t$ have $O_k(t^{k-1})$ choices, while $|V(H)|=k+1$ is fixed,
	so the choices and orderings of vertices in $H$ contribute only $O_k(1)$.
	Hence
	$$
	\begin{aligned}
		N(C_{2k+1},G_m)
		=
		\lceil k^2/2\rceil
		\,k!(k-1)!\binom{t}{k}
		+O_k(t^{k-1})
		=
		\left(
		\frac{\lceil k^2/2\rceil (k-1)!}{(k+1)^k}
		+o(1)
		\right)m^k.
	\end{aligned}
	$$
\end{proof}

\begin{remark}
	If $k$ is odd and $H=K_{k+1}$ minus a perfect matching, then $H$ is
	$(k-1)$-regular and has $(k^2-1)/2$ edges. In this case
	$$
	\rho(H\vee I_t)
	=\frac{k-1+\sqrt{(k-1)^2+4(k+1)t}}2
	<g_k\left((k+1)t+\frac{k^2-1}{2}\right).
	$$
	Thus the graph with one fewer core edge has adjacency spectral radius
	strictly below $g_k\bigl(e(H\vee I_t)\bigr)$. The
	construction in Proposition~\ref{prop:construction} uses
	$(k^2+1)/2$ edges in $H$, explaining the ceiling in the constant
	$\lceil k^2/2\rceil (k-1)!/(k+1)^k$.
\end{remark}

\section{The two bipartite cases}

By the stability result, it remains to consider graphs that are close
to a complete bipartite graph $K_{A_n,B_n}$. We divide the proof into
two cases according to the sizes of the two bipartition classes. If
one class has bounded size, we use a resolvent argument. If both
$|A_n|$ and $|B_n|$ tend to infinity, we use the odd spectral moment
together with the estimate for non-injective closed walks. These two
cases complete the proof after the stability reduction.


We now treat the case where $G$ is close to a complete bipartite graph
with one bipartition class of bounded size. The main difficulty is that
the first-order comparison at the threshold $\theta=g_k(m)$ may have equality
cases. To handle them, we isolate the bounded side and use a resolvent
expansion to obtain a finer spectral estimate.

\begin{proposition}\label{prop:bounded}
Fix $R\ge1$. Let $G_n$ be a sequence of graphs with
$m_n=e(G_n)\to\infty$ and $\rho(G_n)>g_k(m_n)$. Suppose that, after
deleting isolated vertices, there are disjoint sets $S_n,D_n$ such that
$1\le |S_n|\le R$ and
$\dist(G_n,K_{S_n,D_n})=o(m_n)$. Then
$$
N(C_{2k+1},G_n)\ge
\left(
\frac{\lceil k^2/2\rceil (k-1)!}{(k+1)^k}-o(1)
\right)m_n^k.
$$
\end{proposition}

\begin{proof}
Suppose not. Then there are $\varepsilon_0>0$ and a subsequence such that
$$
N(C_{2k+1},G_n)\le
\left(
\frac{\lceil k^2/2\rceil (k-1)!}{(k+1)^k}
-\varepsilon_0
\right)m_n^k.
$$
Passing to a subsequence, we may also assume that
$r:=|S_n|$ and $s:=e(G_n[S_n])$ are fixed.
Let
$D_0:=\{z\in V(G_n)\setminus S_n:N_{G_n}(z)=S_n\}$ and
$H:=G_n-D_0$. Every vertex of $D_n\setminus D_0$ is incident with an
edge of $E(G_n)\triangle E(K_{S_n,D_n})$, and every vertex of
$D_0\setminus D_n$ contributes all its $r$ incident edges to this
symmetric difference. Consequently,
$|D_n\setminus D_0|\le2\dist(G_n,K_{S_n,D_n})$ and
$|D_0\setminus D_n|\le\dist(G_n,K_{S_n,D_n})/r$. Also,
$|m_n-r|D_n||\le\dist(G_n,K_{S_n,D_n})$.

Since $\dist(G_n,K_{S_n,D_n})=o(m_n)$ and $r$ is fixed, these
inequalities give $|D_n\setminus D_0|=o(m_n)$ and
$r|D_0|=(1-o(1))m_n$. The set $D_0$ is independent and complete to
$S_n$. Therefore
$m_n=r|D_0|+e(H)$ and $e(H)=o(m_n)$.
Moreover, by the definition of $D_0$,
$N_H(v)\ne S_n$ for $v\in V(H)\setminus S_n$.

Define the nonnegative defect of $S_n$ by
$$
\operatorname{def}(S_n):=
\sum_{v\in V(H)\setminus S_n}
d_{S_n}(v)\bigl(r-d_{S_n}(v)\bigr)
+r e(H-S_n).
$$
Since
$e(H)-s=\sum_{v\in V(H)\setminus S_n}d_{S_n}(v)+e(H-S_n)$,
we have
$\operatorname{def}(S_n)\le r(e(H)-s)$ and
$e(H)-s\le\left(3+\frac1r\right)\operatorname{def}(S_n)$.
Indeed, the vertices satisfying
$1\le d_{S_n}(v)\le r-1$ contribute at most
$\operatorname{def}(S_n)$ edges to $S_n$. By the definition of $D_0$, every vertex satisfying
$d_{S_n}(v)=r$ is incident with an edge of $H-S_n$, so these vertices
contribute at most $2\operatorname{def}(S_n)$ edges to $S_n$.
Finally, $e(H-S_n)\le\operatorname{def}(S_n)/r$. Adding these three
bounds proves the second defect inequality.

Since $s\le \binom{r}{2}\le R^2$ and $3+1/r\le 4$, the inequality
$e(H)-s\le (3+1/r)\operatorname{def}(S_n)$ implies that
$e(H)\le R^2+4\operatorname{def}(S_n)$.
Moreover, since $m_n=r|D_0|+e(H)$ and
$r|D_0|=(1-o(1))m_n$, we have
$e(H)=m_n-r|D_0|=o(m_n)=o(|D_0|)$, and hence
$e(H)/|D_0|=o(1)$.
Using the inequality
$\operatorname{def}(S_n)\le r(e(H)-s)\le r e(H)$,
we further obtain
$\operatorname{def}(S_n)/|D_0|=o(1)$. Thus
$e(H)\le R^2+4\operatorname{def}(S_n)$,
$e(H)/|D_0|=o(1)$, and
$\operatorname{def}(S_n)/|D_0|=o(1)$.

Let $A_H=A(H)$, and let $\mathbf{c}\in\mathbb R^{V(H)}$ be the characteristic column vector of $S_n$.  Thus $A_H$ has rows and columns indexed by
$V(H)$, and $\mathbf c_u=1$ if $u\in S_n$ and $\mathbf c_u=0$ otherwise.
We have $\mathbf c^{\top}\mathbf c=r$ and
$\mathbf c^{\top}A_H\mathbf c=2s$.
Moreover, we have
$(A_H\mathbf c)_v=d_{S_n}(v)$ for $v\in V(H)$.
Therefore,
$$
\mathbf c^{\top}A_H^2\mathbf c
=\sum_{v\in V(H)}d_{S_n}(v)^2
=\sum_{u\in S_n}d_{S_n}(u)^2+\sum_{v\in V(H)\setminus S_n}d_{S_n}(v)^2.
$$
By the definition of $\operatorname{def}(S_n)$, we get
$$
\begin{aligned}
r(e(H)-s)-\operatorname{def}(S_n)
=
r\!\!\!\sum_{v\in V(H)\setminus S_n}\!\!\!d_{S_n}(v)
\!\!-\!\!\sum_{v\in V(H)\setminus S_n}\!\!\!
d_{S_n}(v)(r-d_{S_n}(v))
=
\!\!\!\sum_{v\in V(H)\setminus S_n}\!\!\!d_{S_n}(v)^2.
\end{aligned}
$$
Hence,
	\begin{equation*}\label{eq:nu-two-general}
		\nu_2:=\mathbf{c}^{\top}A_H^2\mathbf{c}
		=\sum_{u\in S_n}d_{S_n}(u)^2+r(e(H)-s)-\operatorname{def}(S_n).
	\end{equation*}

Let $\theta=g_k(m_n)$.
By
Observation~\ref{obs:threshold-optimization}~(i),
$\theta^2-(k-1)\theta=m_n-\binom{k}{2}$.
The trace-square identity gives
$\rho(H)^2\le2e(H)=o(m_n)$, whereas
$\theta^2=(1+o(1))m_n$ and $K_{r,|D_0|}\subseteq G_n$.
Consequently, for all sufficiently large $n$,
$\theta>\rho(H)$ and $\rho(G_n)>\rho(H)$.

Order the vertices with those in $V(H)$ first and those in $D_0$
second. Then
$$
A(G_n)=
\begin{pmatrix}
A_H&\mathbf c\mathbf 1_{|D_0|}^{\top}\\
\mathbf 1_{|D_0|}\mathbf c^{\top}
&0_{|D_0|\times |D_0|}
\end{pmatrix},
$$
where $\mathbf 1_{|D_0|}$ is the all-ones column vector of length
$|D_0|$.

For $x>\rho(H)$, define
$$
\Phi(x):=
\frac{|D_0|}{x}
\mathbf c^{\top}(xI-A_H)^{-1}\mathbf c.
$$
Taking the Schur complement of $xI-A_H$ gives
$$
\det(xI-A(G_n))
=
\det(xI-A_H)x^{|D_0|-1}
\left(
x-|D_0|\mathbf c^{\top}(xI-A_H)^{-1}\mathbf c
\right).
$$
Since $\rho(G_n)>\rho(H)$, the first determinant does not vanish at
$x=\rho(G_n)$. Hence the final factor must vanish, and therefore
$\Phi(\rho(G_n))=1$.

By Neumann expansion,
$
\Phi(x)=
|D_0|\sum_{j\ge0}
\frac{\mathbf c^{\top}A_H^j\mathbf c}{x^{j+2}}.
$
All coefficients are nonnegative, so $\Phi$ is strictly decreasing.
Retain the first three terms at $x=\theta$:
$$
\Phi_{\le2}(\theta):=
|D_0|
\left(
\frac r{\theta^2}
+\frac{2s}{\theta^3}
+\frac{\nu_2}{\theta^4}
\right).
$$
Using $m_n=r|D_0|+e(H)$ and
$\theta^2-(k-1)\theta=m_n-\binom{k}{2}$, we have
$r|D_0|=\theta^2-(k-1)\theta+\binom{k}{2}-e(H)$.
Substituting this identity and the formula for $\nu_2$ above into
$1-\Phi_{\le2}(\theta)$ gives
\begin{align}\label{eq:general-resolvent-deficit}
1-\Phi_{\le2}(\theta)
={}&
\frac{(k-1)r-2s}{r\theta}\notag\\
&+
\frac{
\operatorname{def}(S_n)
+s(r+2(k-1))
-r\binom{k}{2}
-\sum_{u\in S_n}d_{H[S_n]}(u)^2
}{r\theta^2}\notag\\
&+
\frac{
(k-1)\nu_2
+2s\left(e(H)-\binom{k}{2}\right)
}{r\theta^3}
+
\frac{
\left(e(H)-\binom{k}{2}\right)\nu_2
}{r\theta^4}.
\end{align}

We next estimate the omitted tail. Since
$\|A_H\|\le\sqrt{2e(H)}$, $e(H)=o(m_n)$ and $\theta^2=(1+o(1))m_n$, we have $\|A_H\|/\theta=o(1)$. Notice that $m_n$ is sufficiently large, so we may assume
$\|A_H\|/\theta\le1/2$. Moreover,
$\mathbf c^{\top}A_H^j\mathbf c\le r\|A_H\|^j$. Therefore
\begin{align}
R_{\ge3}(\theta)
&:=
\Phi(\theta)-\Phi_{\le2}(\theta)\notag\\
&\le
\frac{|D_0|r}{\theta^2}
\sum_{j\ge3}
\left(\frac{\|A_H\|}{\theta}\right)^j
\le
\frac{2|D_0|r}{\theta^2}
\left(\frac{\sqrt{2e(H)}}{\theta}\right)^3
\le
C_{k,R}\left(\frac{e(H)}{|D_0|}\right)^{3/2}.
\label{eq:general-tail}
\end{align}
Here and below constants may depend on $k$ and $R$. We also use
$\theta^2\asymp |D_0|$, which follows from
$r|D_0|=(1-o(1))m_n$.

We now distinguish three cases.

\setcounter{case}{0}

\noindent\emph{\bf Case 1: $2s<(k-1)r$.}

Since $(k-1)r-2s$ is a positive integer,
$(k-1)r-2s\ge1$. Hence the first term in \eqref{eq:general-resolvent-deficit} is
at least $1/(r\theta)$. Since
$s(r+2(k-1))\ge0$ and
$\sum_{u\in S_n}d_{H[S_n]}(u)^2\le r(r-1)^2$, the second term in \eqref{eq:general-resolvent-deficit} is at
least
$$
\frac{\operatorname{def}(S_n)}{r\theta^2}
-\frac{\binom{k}{2}+(r-1)^2}{\theta^2}.
$$
Moreover,
$0\le\nu_2\le r\|A_H\|^2\le2r e(H)$. If
$e(H)\ge\binom{k}{2}$, the final two terms in \eqref{eq:general-resolvent-deficit}
are nonnegative. If $e(H)<\binom{k}{2}$, since $s\le\binom{r}{2}$, $r\le R$, $\nu_2\le 2r e(H)<2r\binom{k}{2}$, and $\theta\ge1$ for all sufficiently large $n$, we have
$$
\begin{aligned}
\frac{(k-1)\nu_2+2s\big(e(H)-\binom{k}{2}\big)}{r\theta^3}
+\frac{\big(e(H)-\binom{k}{2}\big)\nu_2}{r\theta^4} 
&\ge{} -\frac{2s\big(\binom{k}{2}-e(H)\big)}{r\theta^3}
-\frac{\big(\binom{k}{2}-e(H)\big)\nu_2}{r\theta^4} \\
&\ge -\frac{\binom{k}{2}}{\theta^3}
\big(\frac{2s}{r}+\frac{\nu_2}{r}\big) 
\ge{} -\frac{\binom{k}{2}\big(R-1+2\binom{k}{2}\big)}{\theta^3}.
\end{aligned}
$$
Thus, in either case, the sum of the last two terms in \eqref{eq:general-resolvent-deficit} is at least $-C_{k,R}/\theta^3$.

Suppose first that
$\operatorname{def}(S_n)\le\sqrt{|D_0|}$. Since
$e(H)\le R^2+4\operatorname{def}(S_n)$, we have
$e(H)=O_R(\sqrt{|D_0|})$. The tail estimate \eqref{eq:general-tail} gives
$
R_{\ge3}(\theta)
=O_{k,R}(|D_0|^{-3/4})
=o(\theta^{-1}).
$
Thus, for all sufficiently large $n$,
$$
1-\Phi_{\le2}(\theta)
\ge
\frac1{r\theta}
-\frac{\binom{k}{2}+(r-1)^2}{\theta^2}
-\frac{C_{k,R}}{\theta^3}
>
R_{\ge3}(\theta).
$$

Suppose next that
$\operatorname{def}(S_n)>\sqrt{|D_0|}$. Then $\operatorname{def}(S_n)\to\infty$. Since
$s(r+2(k-1))\ge0$, $r\le R$, and
$\sum_{u\in S_n}d_{H[S_n]}(u)^2\le r(r-1)^2$, the numerator of the
second term in \eqref{eq:general-resolvent-deficit} is at least
$
\operatorname{def}(S_n)
-R\binom{k}{2}
-R(R-1)^2.
$
Hence, for all sufficiently large $n$, it is at least
$\operatorname{def}(S_n)/2$. Therefore the second term in \eqref{eq:general-resolvent-deficit} is at least
$\operatorname{def}(S_n)/(2r\theta^2)$.

Since $e(H)\le R^2+4\operatorname{def}(S_n)$, we have
$e(H)=O_R(\operatorname{def}(S_n))$. Also,
$\operatorname{def}(S_n)/|D_0|=o(1)$. Hence the tail estimate \eqref{eq:general-tail} gives
$$
\frac{
R_{\ge3}(\theta)
}{
\operatorname{def}(S_n)/|D_0|
}
\le
C_{k,R}
\sqrt{
\frac{\operatorname{def}(S_n)}{|D_0|}
}
=o(1).
$$

Since $\theta^2=(r+o(1))|D_0|$, we have
$
\frac{\operatorname{def}(S_n)}{2r\theta^2}
=\left(\frac1{2r^2}+o(1)\right)
\frac{\operatorname{def}(S_n)}{|D_0|}.
$
Moreover, $\theta^2\asymp |D_0|$ and
$\operatorname{def}(S_n)>\sqrt{|D_0|}$, so
$$
\frac{C_{k,R}/\theta^3}
{\operatorname{def}(S_n)/|D_0|}
=O_{k,R}\Big(
\frac1{\operatorname{def}(S_n)\sqrt{|D_0|}}
\Big)
=o(1).
$$
Thus both $C_{k,R}/\theta^3$ and $R_{\geq3}(\theta)$ are
$o(\operatorname{def}(S_n)/|D_0|)$, whereas
$\operatorname{def}(S_n)/(2r\theta^2)$ has a fixed positive leading
coefficient on this scale. Returning to
\eqref{eq:general-resolvent-deficit}, its first term is positive, its
second term is at least $\operatorname{def}(S_n)/(2r\theta^2)$, and the
sum of its last two terms is at least $-C_{k,R}/\theta^3$. Hence
$$
1-\Phi_{\leq 2}(\theta)
\ge \frac{\operatorname{def}(S_n)}{2r\theta^2}
-\frac{C_{k,R}}{\theta^3}
>R_{\geq 3}(\theta).
$$
Thus the required comparison holds in both subcases, and hence
$\Phi(\theta)<1$. Since $\Phi$ is strictly decreasing and
$\Phi(\rho(G_n))=1$, this gives
$\rho(G_n)<\theta$, a contradiction.

\noindent\emph{\bf Case 2: $2s=(k-1)r$.}

Since $s\le\binom r2$, the equality $2s=(k-1)r$ implies $r\ge k$.
By Lemma~\ref{lem:decaen},
$$
\begin{aligned}
s(r\!\!+\!\!2(k\!\!-\!\!1))
\!\!-\!\!r\binom{k}{2}
\!\!-\!\!\sum_{u\in S_n}d_{S_n}(u)^2
\ge
s\Big(
r\!\!+\!\!2(k-1)
\!\!-\!\!\frac{2s}{r-1}
\!\!-\!\!r\!\!+\!\!2
\Big)
\!\!-\!\!r\binom{k}{2}
=
\frac{(k-1)r(r-k)}{2(r-1)}.
\end{aligned}
$$

If $r=k$, then $H[S_n]=K_k$. If $\operatorname{def}(S_n)=0$, then
$e(H)-s\le(3+1/r)\operatorname{def}(S_n)$ gives
$e(H)=s=\binom{k}{2}$. Hence $H$ consists of $K_k$ and isolated
vertices. Therefore, after deleting isolated vertices,
$
G_n=K_k\vee I_{|D_0|},
m_n=k|D_0|+\binom{k}{2}.
$
The spectral radius of $K_k\vee I_{|D_0|}$ is the positive root of
$
x^2-(k-1)x=k|D_0|
=m_n-\binom{k}{2}.
$
By the definition of $g_k(m_n)$, this gives
$\rho(G_n)=g_k(m_n)$, a contradiction.

In every remaining case, the numerator of the second term in
\eqref{eq:general-resolvent-deficit} is bounded below by
$
C_{k,R}\bigl(1+\operatorname{def}(S_n)\bigr)
$
for some constant $C_{k,R}>0$. Indeed, if $r>k$, the preceding
inequality gives a positive constant contribution, while if $r=k$,
the excluded case implies $\operatorname{def}(S_n)\ge1$.

Furthermore,
$
e(H)\ge s=\frac{(k-1)r}{2}\ge\binom{k}{2},
$
so the final two terms in \eqref{eq:general-resolvent-deficit} are nonnegative.
Since
$
e(H)\le R^2+4\operatorname{def}(S_n)
$
and
$
\operatorname{def}(S_n)/|D_0|=o(1),
$
the tail estimate \eqref{eq:general-tail} gives
$$
\frac{
R_{\ge3}(\theta)
}{
(1+\operatorname{def}(S_n))/|D_0|
}
\le
C_{k,R}
\sqrt{
\frac{1+\operatorname{def}(S_n)}{|D_0|}
}
=o(1).
$$
Since $\theta^2\asymp |D_0|$, the positive second term in \eqref{eq:general-resolvent-deficit} dominates the
tail in \eqref{eq:general-tail}. Hence
$
1-\Phi_{\le2}(\theta)>R_{\ge3}(\theta).
$
Thus $\Phi(\theta)<1$ and $\rho(G_n)<\theta$, a
contradiction.

\noindent\emph{\bf Case 3: $2s>(k-1)r$.}

Since $2s\le r(r-1)$ and $2s>(k-1)r$, we have $r\ge k+1$.
Fix an orientation $u\to v$ of every edge
$uv\in E(H[S_n])$. An ordered $(k-1)$-tuple
$(u_1,\ldots,u_{k-1})$ from $S_n\setminus\{u,v\}$ and an ordered
$k$-tuple $(z_1,\ldots,z_k)$ from $D_0$ determine the cycle
$
u-v-z_1-u_1-z_2-u_2-\cdots-z_{k-1}-u_{k-1}-z_k-u.
$
This cycle has $uv$ as its unique edge inside $S_n$. It therefore
determines $uv$, and the fixed orientation of $uv$ determines the
displayed ordering. Thus no cycle is counted twice.

There are $(k-1)!\binom{r-2}{k-1}$ choices for the first tuple and
$k!\binom{|D_0|}{k}$ choices for the second. Since
$
k!\binom{|D_0|}{k}=(1+o(1))|D_0|^k
$
and
$
m_n=r|D_0|+e(H)
$
with $e(H)=o(m_n)$, we have
$
|D_0|=(1+o(1))m_n/r.
$
Therefore
$$
\begin{aligned}
N(C_{2k+1},G_n)
\ge
s(k-1)!k!
\binom{r-2}{k-1}
\binom{|D_0|}{k}
=
\Big(
\frac{
s(k-1)!\binom{r-2}{k-1}
}{r^k}
-o(1)
\Big)m_n^k.
\end{aligned}
$$

Since $2s>(k-1)r$, we have
$s\ge\lfloor r(k-1)/2\rfloor+1.$
By Observation~\ref{obs:threshold-optimization}~(ii),
$$
\frac{
s(k-1)!\binom{r-2}{k-1}
}{r^k}
\ge
\frac{
\lceil k^2/2\rceil (k-1)!
}{(k+1)^k}.
$$
Consequently,
$$
N(C_{2k+1},G_n)
\ge
\left(
\frac{
\lceil k^2/2\rceil (k-1)!
}{(k+1)^k}
-o(1)
\right)m_n^k,
$$
a contradiction. This proves the
proposition.
\end{proof}


We now consider the case where $G$ is close to a complete bipartite
graph and both bipartition classes tend to infinity. In this regime,
the non-injective closed walks of length $2k+1$ are negligible. We
combine this fact with the odd spectral moment estimate to obtain a
lower bound for $N(C_{2k+1},G)$ which is stronger than the bound
required in Theorem~\ref{thm:main}.

\begin{lemma}\label{lem:moment}
Let $G$ be an $m$-edge graph with adjacency eigenvalues
$\lambda_1\ge\lambda_2\ge\cdots\ge\lambda_N$, where
$\lambda_1=\rho(G)$, and write $A=A(G)$.  Then
$$
 \tr(A^{2k+1})
 =2\lambda_1^{2k-1}(\lambda_1^2-m)
 +\sum_{i=2}^N
 \lambda_i^2(\lambda_1^{2k-1}+\lambda_i^{2k-1}).
$$
Every summand in the sum is nonnegative.
\end{lemma}

\begin{proof}
The trace of $A^{2k+1}$ is the sum of the $(2k+1)$st powers of the
eigenvalues.  Add and subtract
$\lambda_1^{2k-1}\sum_{i=2}^N\lambda_i^2$ to obtain
\begin{align*}
 \tr(A^{2k+1})
 &=\lambda_1^{2k+1}+\sum_{i=2}^N\lambda_i^{2k+1}\\
 &=\lambda_1^{2k+1}
  -\lambda_1^{2k-1}\sum_{i=2}^N\lambda_i^2
  +\sum_{i=2}^N\lambda_i^2
   (\lambda_1^{2k-1}+\lambda_i^{2k-1})\\
 &=2\lambda_1^{2k-1}(\lambda_1^2-m)
  +\sum_{i=2}^N\lambda_i^2
   (\lambda_1^{2k-1}+\lambda_i^{2k-1}),
\end{align*}
where the last equality uses
$\sum_{i=2}^N\lambda_i^2=2m-\lambda_1^2$.
Finally, $|\lambda_i|\le\lambda_1$ and $2k-1$ is odd, so
$\lambda_1^{2k-1}+\lambda_i^{2k-1}\ge0$; multiplying by
$\lambda_i^2\ge0$ proves the last assertion.
\end{proof}

The following proposition gives the spectral estimate that will be
used to count $C_{2k+1}$ when both bipartition classes have orders
tending to infinity.

\begin{proposition}
\label{prop:moment-large}
Let $G_n$ be a sequence of graphs with $e(G_n)=m_n$ such that
$\rho(G_n)>g_k(m_n)$. Suppose there are complete bipartite graphs
$K_{A_n,B_n}$ satisfying
$$
\dist(G_n,K_{A_n,B_n})=o(m_n),
\qquad
\min\{|A_n|,|B_n|\}\longrightarrow\infty.
$$
Then
$$
\tr(A(G_n)^{2k+1})
\ge\bigl((2k+1)(k-1)-o(1)\bigr)m_n^k.
$$
\end{proposition}
\begin{proof}
Let $N=|V(G_n)|$, write the adjacency eigenvalues of $G_n$ as
$\lambda_1\ge\cdots\ge\lambda_N$, where $\lambda_1=\rho(G_n)$, and
write $m=m_n$.  With $\|\cdot\|_F$ denoting the Frobenius norm,
$$
 \|A(G_n)-A(K_{A_n,B_n})\|_F^2
 =2\dist(G_n,K_{A_n,B_n})=o(m).
$$
By Lemma \ref{lem:weyl},
for every $i$,  we have
$$
 \left|\lambda_i(A(G_n))-\lambda_i(A(K_{A_n,B_n}))\right|
 \le \|A(G_n)-A(K_{A_n,B_n})\|
 \le \|A(G_n)-A(K_{A_n,B_n})\|_F=o(\sqrt m).
$$
Since the two nonzero eigenvalues of $K_{A_n,B_n}$ are
$\sqrt{|A_n||B_n|}$ and $-\sqrt{|A_n||B_n|}$, and
$|A_n||B_n|=m+o(m)$, it follows that
$\lambda_1=(1+o(1))\sqrt m$ and
$\lambda_N=-(1+o(1))\sqrt m$.
Write $\mu=-\lambda_N$. Since $\lambda_1>g_k(m)>(k-1)/2$,
Observation~\ref{obs:threshold-optimization}~(i) gives
$
 \lambda_1^2-m>(k-1)\lambda_1-\binom{k}{2}
 =(k-1+o(1))\sqrt m.
$
Consequently the first term in Lemma~\ref{lem:moment} is at least
$
 (2(k-1)-o(1))m^k.
$
Also,
$
 \mu^2\le\sum_{i=2}^N\lambda_i^2=2m-\lambda_1^2,
$
and hence
$$
 \lambda_1-\mu
 =\frac{\lambda_1^2-\mu^2}{\lambda_1+\mu}
 \ge\frac{2(\lambda_1^2-m)}{\lambda_1+\mu}
 \ge k-1-o(1).
$$
The summand corresponding to $\lambda_N=-\mu$ in the second term of
Lemma~\ref{lem:moment} is therefore at least
\begin{align*}
 \mu^2(\lambda_1^{2k-1}-\mu^{2k-1})
 =\mu^2(\lambda_1-\mu)
   \sum_{j=0}^{2k-2}\lambda_1^{2k-2-j}\mu^j
 \ge\bigl((2k-1)(k-1)-o(1)\bigr)m^k.
\end{align*}
Adding the above two lower bounds proves the result.
\end{proof}

Let $W'_{2k+1}(G)$ denote the number of non-injective closed walks of
length $2k+1$ in $G$; that is, the number of closed walks
$v_0v_1\cdots v_{2k}v_0$ for which the vertices
$v_0,v_1,\ldots,v_{2k}$ are not all distinct.
Then
$$
\tr(A(G)^{2k+1})=2(2k+1)N(C_{2k+1},G)+W'_{2k+1}(G).
$$
Indeed, each copy of $C_{2k+1}$ gives $2(2k+1)$ closed walks, according
to the choice of a starting vertex and one of the two directions; all
other closed walks of this length have a repeated vertex and are
counted by $W'_{2k+1}(G)$.
To apply Proposition~\ref{prop:moment-large}, we shall prove that
$W'_{2k+1}(G)=o(m^k)$ whenever $G$ is $o(m)$ edges away from
$K_{A,B}$ and both $|A|$ and $|B|$ tend to infinity.

\begin{lemma}\label{lem:walk-tools}
	The following two statements hold.
	\begin{enumerate}
		\item[\textup{(i)}] If $F$ is a graph with $p$ edges and $j\ge1$, then
		the number of length-$j$ walks in $F$ is
		$O_j\bigl(p^{\lfloor j/2\rfloor+1}\bigr)$.
		\item[\textup{(ii)}] Let $W=v_0v_1\cdots v_{2k}v_0$ be a
		non-injective closed walk.  The graph formed by the distinct vertices
		and edges occurring on $W$ has an edge cover of size at most $k$.
	\end{enumerate}
\end{lemma}

\begin{proof}
	For \textup{(i)}, write a length-$j$ walk as $v_0v_1\cdots v_j$ and
	record every other edge traversed by the walk.  If $j$ is odd, record
	$$
	v_0v_1,\ v_2v_3,\ \ldots,\ v_{j-1}v_j;
	$$
	if $j$ is even, record
	$$
	v_0v_1,\ v_2v_3,\ \ldots,\ v_{j-2}v_{j-1},
	\quad\text{and}\quad v_{j-1}v_j.
	$$
	In either case, we record $q:=\lfloor j/2\rfloor+1$ oriented edges,
	and every vertex occurring in the walk is an endpoint of a recorded
	edge.  Since $F$ has at most $2p$ oriented edges, there are at most
	$(2p)^q$ possible records.  Each walk determines one such record
	uniquely; some records may fail to form a walk, which only makes this an
	upper bound.  Hence the number of length-$j$ walks in $F$ is at most
	$$
	(2p)^q=O_j\bigl(p^{\lfloor j/2\rfloor+1}\bigr).
	$$
	
	For \textup{(ii)}, some vertex occurs at least twice among
	$v_0,v_1,\ldots,v_{2k}$.  Choose one occurrence of such a vertex as
	the starting point and retain the direction of $W$.  We may then assume
	that $v_0=v_i$ for some
	$i\in\{1,\ldots,2k\}$.
	
	Consider the following $k$ edges traversed by $W$:
	$
	v_1v_2,\ v_3v_4,\ \ldots,\ v_{2k-1}v_{2k}.
	$
	Every vertex occurring among $v_1,\ldots,v_{2k}$ is incident with one
	of these edges.  The vertex $v_0$ is also covered, because $v_0=v_i$.
	Thus these edges cover every distinct vertex occurring on $W$; if some
	of the displayed edges coincide, there are fewer than $k$ distinct
	edges.  Consequently, the subgraph consisting of the vertices and edges
	of $W$ has an edge cover of size at most $k$.
\end{proof}
\begin{lemma}
\label{lem:small-volume}
Let $G$ be an $m$-edge graph and $Z\subseteq V(G)$.  The number of
non-injective closed walks of length $2k+1$ that visit $Z$ is at most
$$
 M\Big(\sum_{z\in Z}d_G(z)\Big)m^{k-1},
$$
where $M=M(k)>0$ is a constant depending only on $k$.
\end{lemma}

\begin{proof}
	The assertion is trivial when $m=0$, so assume that $m\ge1$.  By
	changing the initial position of a closed walk cyclically, every walk
	counted in the lemma can be written as
	$$
	W=v_0v_1\cdots v_{2k}v_0
	$$
	with $v_0\in Z$.  Each resulting walk has at most $2k+1$ possible
	original initial positions.  Hence the total number of walks in the
	lemma is at most $2k+1$ times the number of such walks with $v_0\in Z$.
	We now count the latter walks.
	
	Let
	$$
	V_W:=\{v_0,v_1,\ldots,v_{2k}\},\qquad
	E_W:=\{v_iv_{i+1}:0\le i\le 2k\},
	$$
	where $v_{2k+1}:=v_0$, and define $H_W$ by
	$V(H_W):=V_W$ and $E(H_W):=E_W$.  Thus $H_W$ is a spanning subgraph of
	$G[V_W]$.  We regard two closed walks as having the same type if one
	can be obtained from the other by consistently renaming the vertices.
	Since a walk of length $2k+1$ has only $2k+1$ positions, there are only
	$O_k(1)$ types.
	
	Fix one type and use $v_0,v_1,\ldots,v_{2k}$ for its positions.  Let
	$H$ be the graph determined by this type through the construction above.
	By Lemma~\ref{lem:walk-tools}\textup{(ii)}, $H$ has an edge cover $S$ with
	$s:=|S|\le k$.  Since $S$ covers the vertex in position $0$, it contains
	an edge $v_0v_j$ for some $j$.
	
	To construct a walk of this fixed type, first choose
	$v_0\in Z$ and then choose $v_j\in N_G(v_0)$.  There are at most
	$\sum_{z\in Z}d_G(z)$ choices.  For each of the other $s-1$ edges
	$v_av_b\in S$, the ordered adjacent pair $(v_a,v_b)$ can be chosen in
	at most $2m$ ways.  Because $S$ covers every vertex of $H$, consistent
	choices for these ordered pairs determine the actual vertex at every
	position of $W$.  Choices that assign different values to two
	occurrences of the same vertex, identify two vertices that should be
	distinct, or fail to realize another edge of $H$ do not produce a walk
	of the fixed type and hence only decrease the count.  Thus the number
	of walks of the fixed type is at most
	$$
	\Big(\sum_{z\in Z}d_G(z)\Big)(2m)^{s-1}
	\le 2^{k-1}\Big(\sum_{z\in Z}d_G(z)\Big)m^{k-1}.
	$$
	Finally, summing over the $O_k(1)$ types and multiplying the resulting
	bound by $2k+1$ proves the result.
\end{proof}

We first consider the case where the graph consists of a complete
bipartite graph together with only $O(a_n)$ edges inside $A_n$ and
$O(b_n)$ edges inside $B_n$.

\begin{lemma}
	\label{lem:collision-complete}
	Let $G_n$ be a graph with a vertex partition
	$V(G_n)=A_n\cup B_n$.  Let $a_n:=|A_n|$ and $b_n:=|B_n|$, and suppose
	$a_n,b_n\to\infty$.  Suppose further that every edge joining $A_n$ and
	$B_n$ is present in $G_n$ and that
	$$
	e(G_n[A_n])=O(a_n),\qquad e(G_n[B_n])=O(b_n).
	$$
	Then
	$$
	W'_{2k+1}(G_n)=o\bigl((a_nb_n)^k\bigr).
	$$
\end{lemma}

\begin{proof}
	For the rest of the proof, write $A=A_n$, $B=B_n$,
	$F_A=G_n[A_n]$, $F_B=G_n[B_n]$, $a=a_n$, $b=b_n$,
	$p_A=e(F_A)$, and $p_B=e(F_B)$.
	For a closed walk $W=v_0v_1\cdots v_{2k}v_0$, record, in order, the
	part containing each of $v_0,v_1,\ldots,v_{2k}$.  Call the step $v_iv_{i+1}$ internal if its endpoints
	lie in the same part, and cross otherwise.  Every cross step changes
	the part containing the current vertex.  Since the walk returns to its
	initial part, the number of cross steps is even.  Hence the number $j$
	of internal steps is odd.
	
	Suppose first that $j=1$ and that the unique internal step lies in
	$A$.  Once the part containing the vertex at each position is fixed,
	there are $k+1$ positions in $A$ and
	$k$ positions in $B$.  Choose the ordered internal edge in $F_A$ in
	$O(p_A)$ ways.  Its two endpoints are now fixed, while the remaining
	$k-1$ positions in $A$ and the $k$ positions in $B$ have at most
	$a^{k-1}b^k$ choices.  There are only $O_k(1)$ possible such sequences,
	so before imposing non-injectivity the number of walks is
	$O_k(p_Aa^{k-1}b^k)$.
	
	Because $A$ and $B$ are disjoint, a repeated vertex must occur at two
	positions in the same part.  Fixing a pair of equal $A$-positions
	removes one free $A$-choice and gives
	$O_k(p_Aa^{k-2}b^k)$ walks.  Fixing a pair of equal $B$-positions gives
	$O_k(p_Aa^{k-1}b^{k-1})$ walks.  A union bound over the $O_k(1)$ pairs
	of positions, together with $p_A=O(a)$, gives
	$$
	O_k\bigl(p_Aa^{k-2}b^k+p_Aa^{k-1}b^{k-1}\bigr)
	=O_k\Big((ab)^k\Big(\frac1a+\frac1b\Big)\Big).
	$$
	When the unique internal step lies in $B$, the same argument applies
	with $a$ and $b$ interchanged.  Since $a,b\to\infty$, the total
	contribution in the case $j=1$ is $o((ab)^k)$.
	
	Now let $3\le j=2h+1\le2k-1$ and fix the sequence of the parts
	containing $v_0,v_1,\ldots,v_{2k}$.  Read this sequence cyclically,
	and cut it between $v_i$ and $v_{i+1}$ whenever
	$v_iv_{i+1}$ is a cross step.  The resulting blocks are the maximal
	consecutive vertex sequences lying entirely in one part.  For example,
	the cyclic sequence $AAABBABBB$ is cut as
	$AAA\mid BB\mid A\mid BBB$, with one further cut between the final
	$B$ and the initial $A$.  There are
	$$
	(2k+1)-(2h+1)=2(k-h)
	$$
	cross steps.  Cutting a cyclic sequence at these $2(k-h)$ positions
	produces $2(k-h)$ blocks, which alternate between $A$ and $B$.
	Consequently, there are $k-h$ blocks in $A$ and $k-h$ blocks in $B$.
	A block contained in a part $X\in\{A,B\}$ and containing $q$
	internal steps is a length-$q$ walk in $F_X$.  If $2k+1\geq q\ge1$, then
	$e(F_X)=O(|X|)$, and Lemma~\ref{lem:walk-tools}\textup{(i)} gives
	$$
	O_k\bigl(|X|^{\lfloor q/2\rfloor+1}\bigr)
	$$
	possible vertex sequences.  This formula also holds for $q=0$, when
	the block consists of one vertex.
	
	Let $q_1,\ldots,q_{k-h}$ be the numbers of internal steps in the
	$A$-blocks, and let $r_1,\ldots,r_{k-h}$ be the numbers of internal
	steps in the $B$-blocks.  Multiplying the preceding bounds over all
	blocks in each part gives
	\begin{align*}
		\prod_{i=1}^{k-h}O_k\left(a^{\lfloor q_i/2\rfloor+1}\right)
		&=O_k(a^{E_A}),&
		E_A&:=k-h+\sum_i\left\lfloor\frac{q_i}{2}\right\rfloor,\\
		\prod_{i=1}^{k-h}O_k\left(b^{\lfloor r_i/2\rfloor+1}\right)
		&=O_k(b^{E_B}),&
		E_B&:=k-h+\sum_i\left\lfloor\frac{r_i}{2}\right\rfloor.
	\end{align*}
	Once the vertex sequence in every block is chosen, every cross step
	automatically forms an edge of $G$, because all edges between $A$ and
	$B$ are present.  Hence the number of walks with this fixed sequence
	of parts is at most $O_k(a^{E_A}b^{E_B})$.  This
	counts all closed walks with this sequence and hence also bounds the
	non-injective ones.
	
	Let $\ell$ be the number of blocks containing an odd number of
	internal steps.  Since
	$
	\sum_i q_i+\sum_i r_i=2h+1,
	$
	and 
	$\ell$ is odd and positive, we have
	$
	E_A+E_B
	=2(k-h)+\frac{2h+1-\ell}{2}.
	$
	Moreover,
	$$
	\sum_i\left\lfloor\frac{q_i}{2}\right\rfloor
	\le \left\lfloor\frac{\sum_iq_i}{2}\right\rfloor\le h,
	\qquad
	\sum_i\left\lfloor\frac{r_i}{2}\right\rfloor
	\le \left\lfloor\frac{\sum_ir_i}{2}\right\rfloor\le h.
	$$
	It follows that
	$
	E_A\le k$, $E_B\le k$ and 
	$
	(k-E_A)+(k-E_B)=h+\frac{\ell-1}{2}\ge h$.
	Hence
	$$
	\frac{a^{E_A}b^{E_B}}{(ab)^k}
	=\frac{1}{a^{k-E_A}b^{k-E_B}}=o(1),
	$$
	because $h\ge1$ and $a,b\to\infty$.  Each of the $2k+1$ positions of
	the closed walk is assigned to either $A$ or $B$.  Hence there are at
	most $2^{2k+1}$ possible sequences recording the parts containing the
	vertices; since $k$ is fixed, this
	number is $O_k(1)$.  Thus all walks with
	$3\le j\le2k-1$ contribute $o((ab)^k)$.
	
	It remains to consider $j=2k+1$.  Since all $2k+1$ steps are internal,
	the walk stays entirely in $A$ or entirely in $B$.  Let $F$ be the
	graph on $A\cup B$ with $E(F):=E(F_A)\cup E(F_B)$.  Thus the walks in
	this case are precisely the non-injective closed walks of length
	$2k+1$ in $F$.  Let
	$p:=e(F)=p_A+p_B=O(a+b)$.  Taking $Z=V(F)$ in
	Lemma~\ref{lem:small-volume} counts all these walks.  Hence, by the
	handshaking lemma,
	$$
	W'_{2k+1}(F)
	\le M\Big(\sum_{z\in V(F)}d_F(z)\Big)p^{k-1}
	=2Mp^k.
	$$
	Since $a,b\to\infty$,
	$$
	\frac{(a+b)^k}{(ab)^k}
	=\left(\frac1a+\frac1b\right)^k=o(1),
	$$
	and therefore $W'_{2k+1}(F)=O((a+b)^k)=o((ab)^k)$.  Combining
	the three cases proves the lemma.
\end{proof}

The next lemma shows that, when $G_n$ is close to a complete bipartite
graph with both parts large, the non-injective closed walks of length
$2k+1$ can be ignored.

\begin{lemma}\label{lem:collision}
	Fix $k\ge2$ and $C>0$.  Let $G_n$ be a graph sequence, and let
	$A_n,B_n\subseteq V(G_n)$ be disjoint sets satisfying
	$$
	e(G_n)=m_n,\quad
	N(C_{2k+1},G_n)\le C m_n^k,\quad
	\dist(G_n,K_{A_n,B_n})=o(m_n),\quad
	\min\{|A_n|,|B_n|\}\to\infty.
	$$
	Then
	$$
	W'_{2k+1}(G_n)=o(m_n^k).
	$$
\end{lemma}

\begin{proof}
	For the rest of the proof, write $G=G_n$, $A=A_n$, $B=B_n$,
	$a=|A|$, $b=|B|$, and $m=m_n$.  Let
	$
	q=\dist(G,K_{A,B}).
	$
	Since $q=o(m)$,
	\begin{equation}
		ab=(1+o(1))m.
		\label{eq:ab}
	\end{equation}
	Indeed, $K_{A,B}$ has $ab$ edges and
	$|m-ab|\le\dist(G,K_{A,B})=q$.
	Let
	$$
	q_-:=\bigl|\{uv:u\in A,\ v\in B,\ uv\notin E(G)\}\bigr| \ \
\mbox{
	and} \ \
	\eta=\sqrt{\frac qm}+m^{-1/4}.
	$$
	Thus $\eta=o(1)$ and
	\begin{equation}\label{eq:q-eta}
		\frac{q}{\eta m}
		\le \sqrt{\frac qm}=o(1).
	\end{equation}
	Define
	$A_0=\{u\in A:|B\setminus N_G(u)|\le\eta b\}$, and
		$B_0=\{v\in B:|A\setminus N_G(v)|\le\eta a\}$.
	Summing over the vertices in $A$ and in $B$, respectively, gives
	$$
	\sum_{u\in A}|B\setminus N_G(u)|=q_-,
	\qquad
	\sum_{v\in B}|A\setminus N_G(v)|=q_-.
	$$
	Since $q_-\le q$, the definitions of $A_0,B_0$, together with
	\eqref{eq:ab} and \eqref{eq:q-eta}, give
	\begin{equation}
		|A\setminus A_0|\le\frac{q}{\eta b}=o(a),
		\qquad
		|B\setminus B_0|\le\frac{q}{\eta a}=o(b).
		\label{eq:good-sizes}
	\end{equation}
	Let
	$
	Z=V(G)\setminus(A_0\cup B_0).
	$
	Every edge of $G$ incident with a vertex outside $A\cup B$ belongs to
	$E(G)\setminus E(K_{A,B})$.  Moreover, at most $|A\setminus A_0|b$
	edges of $K_{A,B}$ are incident with $A\setminus A_0$, and at most
	$|B\setminus B_0|a$ are incident with $B\setminus B_0$.  Therefore
	\eqref{eq:good-sizes} implies
	\begin{equation}\label{eq:exceptional-volume}
		\sum_{z\in Z}d_G(z)
		\le \sum_{z\in Z}d_{K_{A,B}}(z)
		+2\bigl|E(G)\setminus E(K_{A,B})\bigr|
		\le |A\setminus A_0|b+|B\setminus B_0|a+2q=o(m).
	\end{equation}
	By Lemma~\ref{lem:small-volume}, the non-injective closed walks
	containing at least one vertex of $Z$ contribute $o(m^k)$.
	
	We now consider walks contained in $A_0\cup B_0$.  Let
	$
	p_A=e(G[A_0])$ and $p_B=e(G[B_0])$.
	We first prove
	\begin{equation}\label{eq:good-internal}
		p_A=O(a),\qquad p_B=O(b).
	\end{equation}
	Fix an edge $uv\in E(G[A_0])$.  If all edges between $A$ and $B$ were
	present, we could construct copies of $C_{2k+1}$ containing $uv$ by
	choosing distinct vertices $x_1,\ldots,x_{k-1}\in A\setminus\{u,v\}$
	and $y_1,\ldots,y_k\in B$ and arranging them as
	$
	uvy_1x_1y_2x_2\cdots y_{k-1}x_{k-1}y_ku,
	$
	so that $uv$ is the unique internal edge of each such copy.  Let
	$\mathcal C_{uv}$ denote the family of copies constructed in this way.
	Then
	$$
|\mathcal C_{uv}|=(k-1)!k!\binom{a-2}{k-1}\binom{b}{k}.
$$
	Some of these copies may fail to occur in $G$ because an edge between
	$A$ and $B$ is missing.  We estimate their number in the following two
	cases.
	
	\smallskip
	\noindent\emph{\bf Case 1: the missing edge is incident with $u$ or $v$.}
	Fix $x\in\{u,v\}$ and $y\in B$ with $xy\notin E(G)$.  Among the copies
	of $C_{2k+1}$ constructed above, suppose that $xy$ is an edge.  If
	$x=v$, then $y_1=y$, while if $x=u$, then $y_k=y$.  In either case, it
	remains to choose and order $x_1,\ldots,x_{k-1}$ from
	$A\setminus\{u,v\}$ and the other $k-1$ vertices in $B\setminus\{y\}$.
	Hence the number of such copies is
	$$
	((k-1)!)^2\binom{a-2}{k-1}\binom{b-1}{k-1}
	=O_k(a^{k-1}b^{k-1}).
	$$
	Since
	$u,v\in A_0$, there are at most $2\eta b$ such pairs.  Hence at most
	$$
	2\eta b\cdot O_k(a^{k-1}b^{k-1})
	=O_k(\eta a^{k-1}b^k)
	$$
	of the copies constructed above are lost in Case~1.
	
	\smallskip
	\noindent\emph{\bf Case 2: the missing edge is incident with neither $u$ nor $v$.}
	Fix $x\in A\setminus\{u,v\}$ and $y\in B$ with $xy\notin E(G)$.  If a
	copy in $\mathcal C_{uv}$ contains $xy$, then $x=x_i$ for some
	$1\le i\le k-1$, and $y$ is one of the two neighbours $y_i,y_{i+1}$
	of $x_i$ on the copy.  There are at most $2(k-1)$ choices for these
	positions.  After they are fixed, we choose and order the remaining
	$k-2$ vertices of $A\setminus\{u,v,x\}$ and the remaining $k-1$
	vertices of $B\setminus\{y\}$.  Thus at most
	$$
	2(k-1)(k-2)!(k-1)!
	\binom{a-3}{k-2}\binom{b-1}{k-1}
	=O_k(a^{k-2}b^{k-1})
	$$
	members of $\mathcal C_{uv}$ contain $xy$ as an edge.  There are at
	most $q_-\le q$ such missing pairs, so at most
	$$
	q\cdot O_k(a^{k-2}b^{k-1})
	=O_k(q a^{k-2}b^{k-1})
	$$
	members of $\mathcal C_{uv}$ are lost.  This completes Case~2.
	
	Combining the two cases.
	At most
	$$
	O_k(\eta a^{k-1}b^k)+O_k(q a^{k-2}b^{k-1})
	$$
	members of $\mathcal C_{uv}$ fail to occur in $G$.  Relative to
	$a^{k-1}b^k$, the two losses are
	$
	O_k(\eta)$, $ O_k\Big(\frac{q}{ab}\Big),
	$
	and both are $o(1)$.  Since
	$|\mathcal C_{uv}|=(1-o(1))a^{k-1}b^k$, at least
	$$
	\bigl(1-O_k(\eta)-O_k(q/(ab))-o(1)\bigr)a^{k-1}b^k
	=(1-o(1))a^{k-1}b^k
	$$
	$C_{2k+1}$ in $\mathcal C_{uv}$ occur in $G$, where the
	$o(1)$ term is independent of the choice of $uv$.  The families
	obtained from distinct internal edges are disjoint, because every cycle
	counted in any such family has a unique
	internal edge.  The hypothesis $N(C_{2k+1},G)=O(m^k)$ and
	\eqref{eq:ab} now give
	$$
	p_A(1-o(1))a^{k-1}b^k\le O(a^kb^k),
	$$
	and hence $p_A=O(a)$.  Interchanging $A$ and $B$ proves the other
	part of \eqref{eq:good-internal}.
	
	Recall that $Z=V(G)\setminus(A_0\cup B_0)$, and hence
	$G-Z=G[A_0\cup B_0]$.  Let $\widetilde G$ be the graph obtained from
	$G-Z$ by joining every vertex of $A_0$ to every vertex
	of $B_0$.  Thus every walk of $G-Z$ is a walk of $\widetilde G$.  By
	\eqref{eq:good-sizes},
	$|A_0|=(1-o(1))a$ and
	$|B_0|=(1-o(1))b$, so both sizes tend to infinity; moreover
	\eqref{eq:good-internal} is equivalent to
	$$
	e(\widetilde G[A_0])=O(|A_0|),\qquad
	e(\widetilde G[B_0])=O(|B_0|).
	$$
	Lemma~\ref{lem:collision-complete} therefore shows that the
	non-injective closed walks containing no vertex of $Z$ contribute
	$o((ab)^k)=o(m^k)$.  Adding this to
	the $o(m^k)$ contribution of the walks containing a vertex of $Z$,
	obtained from
	Lemma~\ref{lem:small-volume} and \eqref{eq:exceptional-volume}, proves
	the lemma.
\end{proof}

The next corollary turns the spectral moment estimate into a lower
bound for the number of copies of $C_{2k+1}$.

\begin{corollary}\label{cor:reduction}
Every sequence in Proposition~\ref{prop:moment-large} satisfies
$$
N(C_{2k+1},G_n)\ge
\left(\frac{k-1}{2}-o(1)\right)m_n^k
\ge
\left(
\frac{\lceil k^2/2\rceil (k-1)!}{(k+1)^k}
-o(1)
\right)m_n^k.
$$
\end{corollary}

\begin{proof}
	Suppose that the first inequality in the statement is false.  Then there are
	$\varepsilon_0>0$ and a subsequence, still indexed by $n$, such that
	$$
	N(C_{2k+1},G_n)\le\left(\frac{k-1}{2}-\varepsilon_0\right)m_n^k.
	$$
	Along this subsequence $N(C_{2k+1},G_n)=O(m_n^k)$, so
	Lemma~\ref{lem:collision} gives $W'_{2k+1}(G_n)=o(m_n^k)$.  Therefore,
	by Proposition~\ref{prop:moment-large},
	\begin{align*}
		2(2k+1)N(C_{2k+1},G_n)
		=\tr(A(G_n)^{2k+1})-W'_{2k+1}(G_n)
		\ge\bigl((2k+1)(k-1)-o(1)\bigr)m_n^k.
	\end{align*}
	Dividing by $2(2k+1)$ gives
	$N(C_{2k+1},G_n)\ge((k-1)/2-o(1))m_n^k$, contrary to the choice of the subsequence.  Hence the
	first inequality in the statement holds.
	
	For $k=2$,
	$\lceil k^2/2\rceil (k-1)!/(k+1)^k=2/9<1/2$.
	For $k\ge3$, using $(k-2)!\le k^{k-2}$ and
	$(1+1/k)^k>1+1/k^2$, we obtain
	$$
	\frac{\lceil k^2/2\rceil (k-1)!}{(k+1)^k}
	\le\frac{(k^2+1)(k-1)!}{2(k+1)^k}
	<\frac{k-1}{2}.
	$$
	This proves the second lower bound.
\end{proof}

\section{Proof of the main theorem}

\begin{proof}[\bf Proof of Theorem~\ref{thm:main}]
It suffices to consider
$$
0<\varepsilon<
\frac{\lceil k^2/2\rceil (k-1)!}{(k+1)^k}.
$$
It is enough to prove that there exists $m_0=m_0(k,\varepsilon)$ such
that every graph $G$ with $e(G)=m\ge m_0$ and $\rho(G)>g_k(m)$
satisfies
$$
N(C_{2k+1},G)\ge
\left(
\frac{\lceil k^2/2\rceil(k-1)!}{(k+1)^k}
-\varepsilon
\right)m^k.
$$
Suppose not. Then there is a sequence of graphs $G_n$ with
$m_n=e(G_n)\to\infty$ such that
$\rho(G_n)>g_k(m_n)$ and
$$
N(C_{2k+1},G_n)
<
\left(
\frac{\lceil k^2/2\rceil (k-1)!}{(k+1)^k}
-\varepsilon
\right)m_n^k.
$$
In particular,
$N(C_{2k+1},G_n)=o(m_n^{k+1/2})$. By
Observation~\ref{obs:threshold-optimization}~(i),
$g_k(m_n)>\sqrt{m_n}$ for all sufficiently large $n$, and hence
$\rho(G_n)>\sqrt{m_n}$.

By Lemma~\ref{lem:stability}, there exist disjoint sets $A_n,B_n$
such that
$\dist(G_n,K_{A_n,B_n})=o(m_n)$.
We may delete the isolated vertices of $G_n$ and replace $A_n,B_n$
by their intersections with the remaining vertex set. This does not
change $e(G_n)$, $\rho(G_n)$, or $N(C_{2k+1},G_n)$, and it can only
decrease $\dist(G_n,K_{A_n,B_n})$. We may also interchange $A_n$ and
$B_n$ when necessary.
For all sufficiently large $n$, both $A_n$ and $B_n$ are nonempty.
Indeed, if one part were empty, then $K_{A_n,B_n}$ would have no
edges, so
$\dist(G_n,K_{A_n,B_n})=m_n$, contrary to
$\dist(G_n,K_{A_n,B_n})=o(m_n)$.

If $\min\{|A_n|,|B_n|\}$ does not tend to infinity, then, after
passing to a subsequence and interchanging the two parts if necessary,
we may assume that $|A_n|$ is a fixed positive integer.
By applying
Proposition~\ref{prop:bounded} with $S_n=A_n$ and $D_n=B_n$, we obtain
$$
N(C_{2k+1},G_n)
\ge
\left(
\frac{\lceil k^2/2\rceil (k-1)!}{(k+1)^k}
-o(1)
\right)m_n^k,
$$
a contradiction.

Therefore $\min\{|A_n|,|B_n|\}\to\infty$. By
Corollary~\ref{cor:reduction},
$$
N(C_{2k+1},G_n)
\ge
\left(
\frac{\lceil k^2/2\rceil (k-1)!}{(k+1)^k}
-o(1)
\right)m_n^k,
$$
a contradiction. This proves the lower bound.

Finally, by Proposition~\ref{prop:construction}, for every
sufficiently large $m$, there exists an $m$-edge graph $G$ satisfying
$\rho(G)>g_k(m)$ and
$$
N(C_{2k+1},G)
=
\left(
\frac{\lceil k^2/2\rceil (k-1)!}{(k+1)^k}
+o(1)
\right)m^k.
$$
Therefore the constant
$\lceil k^2/2\rceil (k-1)!/(k+1)^k$ is asymptotically best possible.
\end{proof}

\section*{Acknowledgments}
The second author is supported by the Natural Science Foundation of Shanghai under Grant No. 25ZR1402390,
the third author by the National Key R\&D Program of China (No.~2022YFA1006400), the National Natural Science Foundation of China (No.~12571376) and Shanghai Institute for Mathematics and Interdisciplinary Sciences, SIMIS (ID‑26‑AMS‑003).
During the early exploratory stage of this work,  language‑model‑based tools were used to brainstorm potential proof strategies. 
Outputs from these tools were used solely for informal inspiration, and the authors take full responsibility for all content in the final manuscript.

	




\begin{thebibliography}{99}

\bibitem{Bollobas-Nikiforov-2007} B. Bollob\'as and V. Nikiforov, Cliques and the spectral radius, 
\emph{J. Combin. Theory Ser. B} \textbf{97} (2007), 859--865.

\bibitem{Chen-Li-2026} H. Chen and Y. Li, An edge-spectral supersaturation of Mubayi's theorem for color-critical graphs, 
arXiv:2607.01073 (2026).

\bibitem{Chen-Li-Tang}
H. Chen, Y. Li, Q. Tang, Supersaturation in Nosal graphs: Triangles and books, arXiv:2607.16746 (2026).

\bibitem{deCaen} D. de Caen, An upper bound on the sum of squares of degrees in a graph, 
\emph{Discrete Math.} \textbf{185} (1998), 245--248.

\bibitem{Erdos-Simonovits-1983} P. Erd\H{o}s and M. Simonovits, Supersaturated graphs and hypergraphs, 
\emph{Combinatorica} \textbf{3} (1983), 181--192.

\bibitem{Fang-Li-Lin-Ma-2025} L. Fang, Y. Li, H. Lin and J. Ma, Spectral supersaturation for color-critical graphs, 
arXiv:2512.22482 (2025).

\bibitem{Fang-Lin-Zhai-2026} L. Fang, H. Lin and M. Zhai, Counting color-critical subgraphs under Nikiforov's condition, 
arXiv:2603.14964 (2026).

\bibitem{Fang-Lin-Zhai-2026-tripartite} L. Fang, H. Lin and M. Zhai, Edge-spectral supersaturation for tripartite color-critical graphs, 
arXiv:2608.04485 (2026).

\bibitem{Horn-Johnson-2013} R. A. Horn and C. R. Johnson, \emph{Matrix Analysis}, 2nd ed.,  Cambridge University Press, Cambridge, 2013.

\bibitem{Li-Feng-Peng-2025a} Y. Li, L. Feng and Y. Peng, A spectral Erd\H{o}s--Faudree--Rousseau theorem, 
\emph{J. Graph Theory} \textbf{110} (2025), 408--425.

\bibitem{Li-Feng-Peng-2025b} Y. Li, L. Feng and Y. Peng, Spectral supersaturation: triangles and bowties, 
\emph{European J. Combin.} \textbf{128} (2025), Paper No. 104171.

\bibitem{Li-Lin-Liu-Zhang-2026} Y. Li, W. Lin, H. Liu and S. Zhang, Spectral Sidorenko inequalities and edge-spectral supersaturation, 
arXiv:2605.26614 (2026).

\bibitem{Li-Liu-Zhang-2025-turan} Y. Li, H. Liu and S. Zhang, Edge-spectral Tur\'an theorems for color-critical graphs with applications, 
arXiv:2511.15431 (2025).

\bibitem{Li-Lu-Peng-2024} Y. Li, L. Lu and Y. Peng, A spectral Erd\H{o}s--Rademacher theorem, 
\emph{Adv. in Appl. Math.} \textbf{158} (2024), Paper No. 102720.

\bibitem{Li-Zhai-Shu-2024-cycles} X. Li, M. Zhai and J. Shu, A Brualdi--Hoffman--Tur\'an problem on cycles, 
\emph{European J. Combin.} \textbf{120} (2024), Paper No. 103966.

\bibitem{Mubayi-2010} D. Mubayi, Counting substructures I: color critical graphs, 
\emph{Adv. Math.} \textbf{225} (2010), 2731--2740.

\bibitem{Nikiforov-2002} V. Nikiforov, Some inequalities for the largest eigenvalue of a graph, 
\emph{Combin. Probab. Comput.} \textbf{11} (2002), 179--189.

\bibitem{Nikiforov-2009} V. Nikiforov, A spectral Erd\H{o}s--Stone--Bollob\'as theorem, 
\emph{Combin. Probab. Comput.} \textbf{18} (2009), 455--458.

\bibitem{Nikiforov-2009-saturation} V. Nikiforov, Spectral saturation: inverting the spectral Tur\'an theorem, 
\emph{Electron. J. Combin.} \textbf{16} (2009), Paper No. 33.

\bibitem{Ning-Zhai-2023} B. Ning and M. Zhai, Counting substructures and eigenvalues I: triangles, 
\emph{European J. Combin.} \textbf{110} (2023), Paper No. 103685.

\bibitem{Ning-Zhai-2025} B. Ning and M. Zhai, Counting substructures and eigenvalues II: quadrilaterals, 
\emph{Electron. J. Combin.} \textbf{32} (2025), Paper No. 4.1.

\bibitem{Pikhurko-Yilma-2017} O. Pikhurko and Z.~B. Yilma, Supersaturation problem for color-critical graphs, 
\emph{J. Combin. Theory Ser. B} \textbf{123} (2017), 148--185.

\bibitem{Zhai-Lin-Shu-2021} M. Zhai, H. Lin and J. Shu, Spectral extrema of graphs with fixed size: cycles and complete bipartite graphs, 
\emph{European J. Combin.} \textbf{95} (2021), Paper No. 103322.

\end{thebibliography}
\end{document}